\documentclass[reqno]{amsart}
\usepackage{amsfonts,amsmath,amssymb}
\usepackage{color}
\usepackage{textcomp}
\usepackage{adjustbox}
\usepackage[T1]{fontenc}
\usepackage[utf8]{inputenc}
\usepackage{lmodern}

\usepackage{cancel}

\usepackage{todonotes}

\usepackage[backref=page, colorlinks=true, citecolor=cyan, linkcolor=blue]{hyperref}

\numberwithin{equation}{section}

\newtheorem{thm}{Theorem}[section]
\newtheorem{defi}{Definition}[section]

\newtheorem{lem}{Lemma}[section]

\newtheorem{rem}{Remark}[section]

\newtheorem{asumamos}{Hypothesis}[section]
 
\newcommand{\R}{\mathbb{R}} 
\newcommand{\N}{\mathbb{N}}

\newcommand{\si}{s^{(0)}}

\makeatletter
\@namedef{subjclassname@2020}{%
  \textup{2020} Mathematics Subject Classification}
\makeatother

\title[Uniform persistence and complete analysis of the  periodic case]{Uniform persistence criteria for a variable inputs chemostat model with delayed response in growth and  complete analysis of the periodic case}

\email{mrodriguezcartabia@dm.uba.ar, daniel.sepulveda@utem.cl }
 
\begin{document} 

\maketitle

\centerline{\scshape Mauro Rodriguez Cartabia and Daniel Sepúlveda Oehninger}
\medskip
% {\footnotesize
% % Enter the address of the first author
%  \centerline{IMAS (UBA-CONICET) and Departamento de Matem\'atica}
%    \centerline{Facultad de Ciencias Exactas y Naturales, Universidad de Buenos Aires}
%    \centerline{Ciudad Universitaria, 1428 Buenos Aires, Argentina}
%    }

%\keywords{}

%\subjclass[2020]{}

% [\emph{IEEE Trans. Automat. Control}, Vol. 52, 2007, pp. 852–862]  
\begin{abstract}
We study  a single-species  chemostat model with variable nutrient input and variable dilution rate  with  delayed (fixed) response in growth. The first goal of this article is to prove that persistence  implies uniform persistence. Then we concentrate in the particular case with periodic nutrient input and same periodic dilution with delayed response in growth. We obtain a threshold for either the (uniform) persistence of the model or that 
 the biomass of every solution tends to vanish. 
Furthermore, we prove that persistence is equivalent to the existence of a unique non-trivial periodic solution. We also prove that this solution is attractive. We remark in no  case we need to impose any restrictions on the size of the delay. 
\end{abstract}

\bigskip
 
 
\section{Introduction}

We consider the cultivation of a species of microorganism inside a chemostat under a limiting substrate with a delay between the consuption and the growth of the population in a variable environment, i.e., both the dilution rate and the input concentration of the substrate vary in time. The need to consider a delay in biomass growth in a chemostat has been documented in the work of Caperon \cite{caperon1969time} and Ellermeyer et al. \cite{ellermeyer2003theoretical}. Furthermore, any population is affected by time-varying environmental fluctuations such as the light cycle or the seasons of the year, which reaffirms the importance of studying non-autonomous population models,   see the book of Smith and Waltman \cite[Chapter 7]{smith1995theory} for instance.

A (periodic version) model suitable for the situation described above has been studied and deduced in  \cite{amster2020existence} by Amster, Robledo and Sep\'ulveda. In the present article, first we study  the model with general inputs, and then imposing periodicity conditions. Therefore, fix a non-negative constant $\tau$  and consider the system
\begin{equation}\label{Main System}
    \begin{array}{rcll}
         s'(t)& =  & D(t)s^{(0)}(t)-D(t)s(t)- p(s(t))x(t), &t\geq 0, \\ \\
         x'(t)& =  & x(t-\tau)p(s(t-\tau))e^{-\int_{t-\tau}^{t}
         D(h)dh}-D(t)x(t), & t\geq 0 
    \end{array}
\end{equation}
with initial conditions
$$\left(s,x \right)\Big|_{[-\tau,0]}=\left(s^{in},x^{in} \right).$$
These initial conditions must be non-negative time functions defined over the interval $[-\tau,0]$. Here, $s(t)$ and $x(t)$ represent, respectively, the substrate and biomass densities inside the bioreactor at time $t$; the dilution rate and the nutrient input concentration, respectively $D(t)$ and $s^{(0)}(t)$, are 
non-negative, continuous, and bounded functions for $t\geq 0$.  As usual in this class of models, the relationship between substrate consumption and biomass growth is modeled by a specific consumption function $p:[0,\infty)\to [0,\infty)$, which satisfies:
%\begin{enumerate}
%    \item[\textbf{(A1)}] $p$ is of class $C^1$, $p'(s)>0$ for each  $s\geq 0$ and $p(0)=0$. 
%\end{enumerate}
\begin{asumamos}\label{A1}
$p$ is of class $C^1$, $p'(s)>0$ for each  $s\geq 0$ and $p(0)=0$. 
\end{asumamos}

%Se mencionan donde se han centrado las investigaciones recientes que abordan problemas en el área.

To begin with, we concentrate on uniform  persistence. The concept of persistence has significant importance in the theory of population models, was introduced in \cite{B-F-W86,F-W77},  and its relevance has increased due to the interest it arouses among those who study ecology and dynamical systems. Ellermeyer in  \cite{ellermeyer1994competition} proved persistence criteria for System (\ref{Main System})  in the   autonomous case, and Ellermeyer  et al. in \cite{ellermeyer2001persistence} provided persistence criteria for System (\ref{Main System}) in the non-autonomous case with instantaneous response in growth ($\tau=0$).  Recently, Rodriguez Cartabia in \cite{cartabia2022persistence} studied the System (\ref{Main System}) obtaining necessary and sufficient conditions for the  persistence of the microbial population.
%Se devela la existencia de una brecha en las investigaciones previas (hacia donde apunta el trabajo).
 Among the several notions of persistence, the uniform persistence is a more desirable from the point of views of applications since is a more robust concept, see \cite[Introduction]{B-F-W86}. We have identified a lack of studies focusing on this topic for System (\ref{Main System}).
 %Nevertheless,   it is notorious the lack of studies focused on this subject for the System (\ref{Main System}), this can be explained by the difficulties of considering the  non-autonomous case with delay.
%\todo{\textcolor{blue}{¿hay alguno? ¿o cambiamos por "there is a lack of study..."? Para este modelo yo sólo conozco dos trabajos enfocados en la persistencia, el tuyo y el de Ye, Zhang \& Teng.} \textcolor{red}{El único que yo conozco es el de "nuestro amigo Z. Teng"  ver \cite{ye2022dynamical}.  No esta disponible, de todos modos creo que sería bueno citarlo para reflejar que el tema y el modelo es de mucho interés en la actualidad.} }
%Se describe el problema que se abordará  en el trabajo y/o se detalla el objetivo de este. 
Therefore, to the best of our knowledge, Theorem \ref{persistenciauniforme} is the first which states that persistence of the System (\ref{Main System}) implies uniform persistence. Roughly speaking, there exists an intrinsic bound $\delta>0$ such that every solution $(s,x)$ of (\ref{Main System}) with not null initial condition (see Definition \ref{condicionesiniciales}) satisfies that $x \geq \delta$ from a certain time. %The key to this result is that the necessary and sufficient conditions for the persistence of  (\ref{Main System}) given in \cite{cartabia2022persistence} imply uniform persistence. 

Secondly, we focus on the case where $\si(t)$ and $D(t)$ are periodic functions. In \cite{amster2020existence} the authors  obtained sufficient conditions for the existence of periodic solutions using the generalized Leray-Schauder degree continuation theorem and, applying the implicit function theorem, proved that for small delays the non-trivial periodic solution  is unique. We emphasize that this  results are based on stronger assumptions about System (\ref{Main System})  than needed.  Therefore, to the best of our knowledge, the results of the present paper are the first comprehensive study of the periodic case. Theorem \ref{teoremadelaextincion} states that if System (\ref{Main System}) is not persistent then  all solutions tend to washout. In other words, obtain a threshold for the vanishing of biomass. Next, we concentrate on the problem of the existence of a positive periodic solution.  Theorem \ref{teoremadeexistencia} establishes that if System (\ref{Main System}) is persistent then there is a non-trivial periodic solution.  The key is to prove this result by combining Horn's  fixed point Theorem   \cite[Theorem 6]{horn1970some}  with Theorem \ref{persistenciauniforme}.  We  finish this article with Theorem \ref{teoremadeestabilidad} which states that this positive periodic solution is unique and attractive. In conclusion, we prove that System (\ref{Main System}) is persistent if and only if there exists a unique  non-trivial attractive periodic solution.

Finally, we emphasize  that, unlike several results in delay differential equations, all proof presented in this article  do not need to impose any restrictions on the fixed delay, i.e., all statements are valid regardless the size of the delay. 
%That is, our results provide a complete study of the dynamics for the $\omega$-periodic version of (\ref{Main System})  and the existence of periodic solutions. 

The rest of the paper is organized as follows. In Section \ref{resultados} we present in more detail previous research on the subject, introduce definitions and present the theorems.  In Section \ref{preliares} we introduce results from previous works and different lemmas needed for the proofs. In Section \ref{pruebapersistencia} we prove Theorem \ref{persistenciauniforme} and, finally, in Section \ref{pruebasperiodico} we present the proofs of Theorems \ref{teoremadelaextincion}, \ref{teoremadeexistencia}, and \ref{teoremadeestabilidad}.

%\textcolor{orange}{A veces notamos $x$, otras $x(t)$ y lo mismo para $D$ y $\si$. ¿Por qué subrayado en vez de emph? A veces aparece (strong) y otras no.}

\section{The results}\label{resultados}

%\cite{ye2022dynamical}

 The introduction of the classical chemostat model has attracted the attention of the mathematical community  which has used it to investigate control, interspecies competition, and persistence, among other problems.  A better understanding of the continuous stirred tank reactor has led to several modifications of the classical model.
 %, with emphasis on those that consider a delay in the microbial growth rate \cite{ellermeyer1994competition,freedman1989chemostat}, time-varying environmental conditions \cite{caraballo2015dynamics,ellermeyer2001persistence}, and recently both situations at the same time \cite{caraballo2015nonautonomous}.
To provide a   brief review of previous research, firstly consider the work of Caperon \cite{caperon1969time} where experimental evidence is reported on the presence of a delay between nutrient consumption and biomass growth, taking into account this delay, the author proposed a model similar to the following system:
\begin{equation*}%\label{Caperon}
    \begin{array}{rcl}
         s'(t)& =  & Ds^{(0)}-Ds(t)- p(s(t))x(t), \\ \\
         x'(t)& =  & x(t)p(s(t-\tau))-Dx(t). 
    \end{array}
\end{equation*}
Amster, Robledo and Sep\'ulveda \cite{amster2020dynamics}  consider a version of this model with periodic substrate concentration  and obtained a necessary and sufficient condition for the existence of a positive periodic solution. On another hand, it is worth mentioning that Caraballo et al.  modified  this system in \cite{caraballo2015nonautonomous} to introduce variable (bounded) delay and  incorporated   the death of the microorganisms in addition to the washout.
 
Another approach to modeling the presence of delay in the growth of a species in a chemostat was carried out by Freedman et al. \cite{freedman1989chemostat} and by Ellermeyer \cite{ellermeyer1994competition,ellermeyer2003theoretical}, who proposed the following system:
\begin{equation*}%\label{Caperon}
    \begin{array}{rcl}
         s'(t)& =  & Ds^{(0)}-Ds(t)- p(s(t))x(t), \\ \\
         x'(t)& =  & x(t-\tau)p(s(t-\tau))e^{-D\tau}-Dx(t). 
    \end{array}
\end{equation*}
%\textcolor{red}{Within the study of non-autonomous chemostat models with delay, we can mention the works of Caraballo et al. \cite{caraballo2015nonautonomous}, and Amster, Robledo and Sep\'ulveda \cite{amster2020dynamics} where they consider versions of the Caperon model with time-varying coefficients. In \cite{caraballo2015nonautonomous} only the delay was considered to be time-varying, whereas in \cite{amster2020dynamics} an $\omega$-periodic substrate concentration and a constant delay were considered. On the other hand, non-autonomous versions of the Ellermeyer model have been studied by Amster, Robledo, and Sep\'ulveda \cite{amster2020existence} and Rodriguez Cartabia \cite{cartabia2022persistence}.}
Note that System (\ref{Main System}) is an extension that incorporates variable nutrient input and variable dilution rate.  One of the difficulties encountered when studying the persistence of a system of differential equations with delay is that the state space is not locally compact so it is necessary to look for new approaches to determine persistence, see \cite{amster2021persistence}. To study persistence, in \cite{cartabia2022persistence} the author extended the criteria given in \cite[Theorem 3]{ellermeyer2001persistence} to incorporate the case with fixed delay in growth. He provided a necessary and sufficient criteria (see Theorem \ref{persistencia}) for the persistence of System (\ref{Main System}). Therefore, we propose the present article as a continuation since our first goal is to show that this persistence criteria also imply uniform persistence. 

We conclude this subsection by pointing out that in the rest of this paper we assume: 
%\begin{enumerate}
%    \item[\textbf{(A2)}]\label{hip2} $\si$ is upper and lower bounded by positive constants, $D$ is non-negative and upper bounded by a positive constant, and the integral of $D$ diverges.
%\end{enumerate}
\begin{asumamos}\label{hip2} $\si(t)$ is upper and lower bounded by positive constants, $D(t)$ is non-negative and upper bounded by a positive constant, and the integral of $D(t)$ diverges.
\end{asumamos}

%In \cite{amster2020dynamics} they obtained necessary and sufficient conditions for the existence of periodic solutions, also determined a condition that guarantees the extinction of the microbial population, and concluded the work proving that for small delays the non-trivial periodic solution of the studied System is unique.

\subsection{Main definitions}   Before presenting the theorems obtained in this work, we introduce the definitions involved in this article.  %For readers familiar with the subject, we remark that all persistence definitions presented in this paper are the strong version.

%For readers familiar with the subject, we remark that every time we mention persistence in this work we refer to the concept of strong persistence.

\begin{defi}[Not null initial conditions]\label{condicionesiniciales} 
We say  $(s^{in},x^{in})$ is a \emph{not null initial condition} if its  time functions are non-negative, and  either $x^{in}(0)>0$ or there exists $t_*\in[-\tau,0]$ such that $s^{in}(t_*)>0$ and $x^{in}(t_*)>0$. 
\end{defi}
Given a solution $(s,x)$ of System (\ref{Main System})  we denote 
\begin{equation}\label{notacioncondicionesiniciales}
    (s,x)(t)=(s,x)(t,(s^{in},x^{in}))
\end{equation}
when the initial condition is fixed.

\begin{defi}[Persistence definitions]\label{def: 3.1}
 The System (\ref{Main System}) is called \emph{(strong) persistent}, if  
$$
\liminf_{t\to\infty} x(t,(s^{in},x^{in}) )>0, \quad \textnormal{for all not null } (s^{in},x^{in}).
$$
The System (\ref{Main System}) is called \emph{uniformly  persistent}, if there exists some  $\delta> 0$ such that
$$
\liminf_{t\to\infty} x(t,(s^{in},x^{in}) )>\delta, \quad \textnormal{for all not null } (s^{in},x^{in}).$$

\end{defi}

Note that in the absence of biomass the System (\ref{Main System}) becomes the linear differential equation
\begin{equation}\label{z}   
  z'(t) =D(t)\left( \si(t)-z(t)\right).    
\end{equation}
We emphasize that any solution of the Equation (\ref{z}) with positive initial condition $z_0$ is also positive for all $t\geq 0$ . This is easily seen by writing a solution in the form
\begin{displaymath}
    z(t)=e^{-\int_0^tD(r)\, dr}z_0+\int_0^{t}e^{-\int_r^tD(r)\, dr}\si(r)dr.
\end{displaymath}
Moreover, for $D(t)$ and $\si(t)$ bounded and continuous functions, we have that every solution $z(t)$ of (\ref{z}) verifies that:
\begin{displaymath}
    \lim_{t\to+\infty}(z(t)-z^*(t))=0,
\end{displaymath}
with $z^*(t)$  the unique bounded solution of (\ref{z}) over $\mathbb{R}$, defined by
\begin{equation}\label{z*}
    z^{*}(t):=\int_{-\infty}^{t}e^{-\int_h^tD(r)\,dr}D(h)s^{(0)}(h)dh.
\end{equation}
We call $(z^*,0)$ the \emph{washout solution} and, for simplicity, sometimes we only refer $z^*$ as the {washout solution.} 

\begin{defi}[Extinction]\label{def: 3.2} A solution $(s,x)$ of System (\ref{Main System}) tends to \emph{extinction} if
$$
\lim_{t\to\infty} x(t )=0.
$$
\end{defi}

\subsection{Results for the general non-autonomous model} 
We now state the first main result of this paper. It is worth recalling that throughout this note we consider $\tau \geq 0$ and, in particular, all results apply to the  case without delay. 

\begin{thm}[Uniform persistence]\label{persistenciauniforme}
The System  (\ref{Main System}) is persistent if and only if it is uniform persistent. 

Furthermore, if the system is persistent then there exists $\delta >0$ such that for all $R>0$ and $\alpha >0$ there is $T^{in}=T^{in}(R,\alpha)\geq 0$  such that for any solution $(s,x)$ with not null initial conditions satisfying $|| (s^{in},x^{in}) ||\leq R$ and such that $x(\tau)\geq \alpha$, then $ x(t)\geq \delta$  for all $t\geq T^{in}.$
\end{thm} 

\subsection{Results for the periodic model} Fix a constant $\omega>0$. 
 Considering that both the dilution rate $D(t)$ and the input nutrient concentration $\si(t)$ are positive $\omega$-periodic functions, we wonder whether persistence is a necessary and sufficient condition for the existence of a positive $\omega$-periodic solution of the System (\ref{Main System}).

\begin{thm}[Extinction]\label{teoremadelaextincion}
If the periodic version of System (\ref{Main System})
is not persistent then every solution tends to extinction.  Namely, in this case every solution tends to  the washout solution.
\end{thm}

Note that the converse of the latter result is trivial.
 
\begin{thm}[Existence of periodic solution]\label{teoremadeexistencia}
If  the System (\ref{Main System}) is persistent then there is at least one positive $\omega$-periodic solution.
\end{thm}

\begin{thm}[Attractivity of periodic solution]\label{teoremadeestabilidad}
Under the assumption of the Theorem \ref{teoremadeexistencia} there exists a unique positive $\omega$-periodic 
solution that exponentially attracts every solution with not null initial conditions.
\end{thm}

 \begin{rem}[Complete analysis of the periodic case] By Theorem \ref{teoremadelaextincion}, observe that if System (\ref{Main System})  is not persistent then there is no  positive $\omega$-periodic solution. Therefore, combining all results we get that System (\ref{Main System}) is (uniform) persistent if and only if there is an unique attractive positive $\omega$-periodic solution.  
\end{rem}

\section{Preliminaries}\label{preliares}

\subsection{Preliminaries}
%In the absence of biomass the System (\ref{Main System}) becomes the linear differential equation
%\begin{equation}\label{z}  
%  z'(t) =D(t)\left( \si(t)-z(t)\right).    
%\end{equation}
%We emphasize that any solution of the System (\ref{z}) with positive initial condition $z_0$ is also positive for all $t\geq 0$ . This is easily seen by writing a solution in the form
%\begin{displaymath}
%    z(t)=e^{-\int_0^tD}z_0+\int_0^{t}e^{-\int_r^tD}\si(r)dr.
%\end{displaymath}
%Moreover, \textcolor{blue}{for $D$ and $\si$ bounded, positive and continuous functions (no, $D$ puede valer cero en algún momento)}, we have that every solution $z(t)$ of (\ref{z}) verifies that:
%\begin{displaymath}
%    \lim_{t\to+\infty}(z(t)-z^*(t))=0,
%\end{displaymath}
%with $z^*(t)$  the unique bounded solution of (\ref{z}) over $\mathbb{R}$, defined by
%\begin{equation}\label{z*}
%    z^{*}(t)=\int_{-\infty}^{t}e^{-\int_h^tD(\sigma)d\sigma}D(h)s^{(0)}(h)dh.
%\end{equation}
The deduction of the model given by System (\ref{Main System}) is based on considering a function $y(t)$ that represents the amount of substrate that has been absorbed by the biomass during the time interval $[t-\tau,t]$ and that remains in the bioreactor at the instant $t$, this function is given by:  
\begin{equation}\label{funciony}
  y(t):=\int_{t-\tau}^t x(h)p(s(h))e^{-\int_h^tD(r)\,dr}\,dh.
\end{equation}
A key fact in achieving the results of this work lies on the relationship between the solutions of (\ref{Main System}) and the functions $z^*(t)$ and $y(t)$.

Next, we define  
\begin{equation*}\label{funcionc}
  c(t):=c(0)e^{-\int_0^tD(r)\,dr}+\int_{-\tau}^{t-\tau}c(h)p(z^*(h))e^{-\int_h^tD(r)\,dr}\,dh,\quad  t\geq 0
\end{equation*} 
and we assume that  $c(\theta)\geq 0$ for all $\theta\in[-\tau,0]$  with $c(0)>0$ which is a solution of the linear equation 
$$c'(t)=-D(t)c(t)+c(t-\tau)p(z^*(t-\tau))e^{-\int_{t-\tau}^tD(r)\,dr}.$$
Since $c(t)>0$ for all $t\geq 0$ we can define the function 
\begin{equation}\label{funcionphi}
  \varphi(t):=\frac{c(t)}{c(t+\tau)}e^{-\int_t^{t+\tau}D(r)\,dr},\quad t\geq 0 .
\end{equation}
This function, which is independent of any solution $(s,x)(t)$, is inherent in the System (\ref{Main System}) and is the key to determine the persistence.  It is explicitly given by 
\begin{equation*} 
  \varphi(t) = \frac{c(0)e^{-\int_0^tD(r)\,dr}+\int_{-\tau}^{t-\tau}c(h)p(z^*(h))e^{-\int_h^tD(r)\,dr}\,dh}{c(0)e^{-\int_0^tD(r)\,dr}+\int_{-\tau}^{t}c(h)p(z^*(h))e^{-\int_h^tD(r)\,dr}\,dh}.
\end{equation*}
Observe that multiples of $c$ result in the same $\varphi$ and that the image of $\varphi$ is contained in $(0,1]$. 

On the other hand, but related with $\varphi(t)$, given a solution $(s,x)(t)$ of the System (\ref{Main System}) with not null initial conditions, we define a function $\psi(t)=\psi(s,x)(t)$ by 
\begin{equation}\label{eq: 4.2}
    \psi(t):=\frac{x(t)}{x(t+\tau)}e^{-\int_{t}^{t+\tau}D(r)dr},\quad t\geq 0.
\end{equation}
 Using (\ref{eq: 4.2}) in the second equation of (\ref{Main System}) we obtain
\begin{equation}\label{eq: 3.7}
 x(t+\tau)=x(t_0)e^{\int_{t_0}^{t+\tau}\left[p(s(h-\tau))\psi(h-\tau)-D(h)\right]dh},\quad t\geq t_0\geq \tau.
\end{equation}
In addition, by using the previous equation with $t_0=t$ in Equation (\ref{eq: 4.2})  we obtain
\begin{equation}\label{psi}
    \psi(t)=e^{-\int_{t-\tau}^{t}p(s(h))\psi(h)dh},\quad t\geq  \tau.
\end{equation}

We remark that, through  this article, in a sum involving $z^*(t)$, $s(t)$, $x(t)$ or $y(t)$ we simplify the dependence of time. For example, we note $(x+y)(t)$ instead of $x(t)+y(t)$.

Let us introduce additional notation to be used in this note. We consider the Banach space    $  \mathcal{C}:=C([-\tau,0]\to\R^2)$ with the norm
$$\|\phi\|=\max_{t\in[-\tau,0]}|\phi(t)|=\max_{t\in[-\tau,0]}\sqrt{\phi_1^2(t)+\phi_2^2(t)}.$$
As usual, for a given continuous function $\phi:[-\tau,\infty)\to \R^2$ and any $t\geq 0$ we define $\phi_t\in \mathcal{C}$  as  
\begin{equation}\label{definicionenC}
\phi_t(h)=\phi(t+h)
\end{equation}
for $h\in[-\tau,0]$. Finally, for an $\omega$-periodic function $f:\mathbb{R}\to\mathbb{R}$ we denote its average as 
$$\langle f\rangle:=\frac{1}{\omega}\int_0^\omega f(t)\, dt.$$

\subsection{Preliminaries results}

Recall Hypothesis    \ref{hip2} and then consider that  $\si $ and $D $ are bounded above by positive constants $\overline s$ and $\overline D$, respectively, and $\si $ is bounded below by a positive  constant $\underline s$. Note that a consequence of  the definition of $z^* $ in (\ref{z*}) 
is that  $z^*(t)\leq \overline s$ for all $t$.  

An alternative formulation of a  persistence criterion for the System (\ref{Main System}) is presented below. To the proof see \cite[Theorem 2.1]{cartabia2022persistence}.

\begin{thm}[Persistence criteria]\label{persistencia} 
Let $\tau$ be any non-negative constant  and   let $z^*(t)$ and $\varphi(t)$ be the functions defined by (\ref{z*}) and  (\ref{funcionphi}), respectively.

Therefore  System  (\ref{Main System}) is persistent if and only if there are positive constants   $\eta$ and $T$ such that 
\begin{equation}\label{condicion}
  \int_{t_1}^{t_2} p(z^*(t-\tau)) \varphi(t-\tau)\, dt> \int_{t_1}^{t_2}\left( D(t)+\eta\right)  \, dt
\end{equation}
for all $t_1> T$, $t_2-t_1> T$.
\end{thm}

 For the the particular case when $\si(t)$ and $D(t)$ are $\omega$-periodic we have the following result.

\begin{thm}[Persistence for $\omega$-periodic case]\label{persistenciaperiodica} 
Let $\tau$ be any non-negative constant and assume the functions  $\si(t)$ and  $D(t)$ are $\omega$-periodic. Then $z^*(t)$ is $\omega$-periodic and there is a unique $c(t)$ (up to a constant factor) such that $\varphi(t)$ defined  in   \eqref{funcionphi} is $\omega$-periodic.
Therefore  System  (\ref{Main System}) is persistent if and only if 
\begin{equation*} 
   \langle p(z^* ) \varphi \rangle>\langle  D\rangle.
\end{equation*}
\end{thm}

To the proof see \cite[Theorem 3.1]{cartabia2022persistence}.

In order to present necessary lemmas, we now introduce a function  $f :[-\tau,\infty)\to\mathbb{R}^+$ bounded  from above by a positive constant $M$ and  such that 
$$0<\inf \limits_{t\geq -\tau} f(t).$$  
Moreover, we assume that $M$ satisfies
\begin{equation}\label{cotaM}
  M\geq \frac{1}{4\tau+1}.
\end{equation} 
Also consider a function $g :[-\tau,\infty)\to\R^+$ which verifies the property that there exists $t_g\geq \tau $ such that 
\begin{equation}\label{condiciong}
  g(t)\leq M,\quad t\geq t_g-\tau.
\end{equation}

We conclude this section by stating four lemmas that are fundamental in the proof of our uniform persistence result. The proves can be found in \cite[Lemma 4.1, Lemma 4.2, Lemma 4.4, and Lemma 4.5]{cartabia2022persistence}, respectively. We observe that the last  proof is the reason to ask for (\ref{cotaM}).

\begin{lem}\label{funcionz} 
The functions  $z^*(t),\,s(t),\,x(t)$ and  $y(t)$ given in (\ref{z*}), (\ref{Main System}) and (\ref{funciony}), respectively, satisfy
$$ (z^* -s -x -y)(t) = \left(z^* -s -x -y\right)(0) e^{-\int_0^tD(r)\, dr}.$$
Furthermore, they all are bounded above.
\end{lem}
 
\begin{lem}\label{propiedadphi}
 The function $\varphi(t)$ defined in \eqref{funcionphi} satisfies
$$\varphi(t)=e^{-\int_{t-\tau}^t\varphi(h)p(z^*(h))\,dh}.$$
\end{lem}

\begin{lem}\label{hermana}
Let $0\leq t_0<t_1$, $\tau>0$  and $\varphi:[t_0-\tau,\infty)\to(0,1]$ be such that 
$$\varphi(t)=e^{-\int_{t-\tau}^tf(h)\varphi(h)\, dh}$$
for all $t>t_0.$
Assume that there is  $\varepsilon>0$ such that    $|f(t)-g(t)|<\varepsilon$  for all $t\in [t_0,t_1]$. Then  there exists  $\tilde \psi:[t_0-\tau,\infty)\to(0,1]$ that satisfies 
$$\tilde \psi(t)=e^{-\int_{t-\tau}^tg(h)\tilde \psi(h)\, dh} $$
for all $t>t_0$ and such that 
 $$|\varphi(t)-\tilde \psi(t)|<2\varepsilon\tau e^{2M(t-t_0)} $$ 
for all $t\in [t_0-\tau,t_1]$.
\end{lem}
 
\begin{lem}\label{lemacrucial} 
Let  $t_0\geq t_g$, $\tau> 0$ and consider functions $\varphi,\,\psi:[t_0-\tau,\infty)\to(0,1]$  satisfying 
\begin{align*}
  \varphi(t)&=e^{-\int^t_{t-\tau}f(h)\varphi(h)\, dh},\\
  \psi(t)&=e^{-\int^t_{t-\tau}f(h)\psi(h)\, dh}
\end{align*}
for all $t\geq t_0$. Therefore  
$$|\varphi(t)-\psi(t)|< 3M \sqrt{\frac{ (t-t_0)}{ \inf f}} \left(1-e^{-M\tau}\right)^{(t-t_0)/(2\tau)-1/2} $$
for all $t\geq t_0 $ where $\inf f=\inf_{t\in [-\tau,\infty) }f(t)$.
\end{lem}

\section{Proof of uniform persistence for the non-autonomous case}\label{pruebapersistencia}

\subsection{Ideas for the proof}

To prove Theorem \ref{persistenciauniforme} we refine the proof of \cite[Theorem 2.1]{cartabia2022persistence}. Observe that if System (\ref{Main System}) is persistent then it satisfies condition given by Inequality (\ref{condicion}). To prove uniform persistence by contradiction, consider $x$ small enough. Then $y$ is also small and then $s\approx z^*$ (see Lemma \ref{funcionz}) which implies that $\varphi \approx \psi $ and  
$$\int p(z^*)\varphi\approx \int p(s)\psi, $$
which is larger than the integral of $D$ if we assume  Inequality (\ref{condicion}). Finally, by Equation (\ref{eq: 3.7}) $x$ can not tend to zero when the time goes to infinity and we get the contradiction.

Furthermore, note that $x$ at the begging might be extremely small. Then we firstly prove that there is a time $T^{in}_*$ (that depends on the size of the norm of the initial condition  $(s^{in},x^{in})$ and the value of $x$ at time $\tau$) for which we can ensure that $x$ grows \emph{enough} (see Inequality (\ref{testrella})). Finally, once that $x$ is large enough we can prove that it is \emph{away from zero} for all remaining positive times.

\subsection{Auxiliary lemma for the non-autonomous chemostat model}

The following result is a convenient adaptation of  \cite[Lema 4.5]{cartabia2022persistence}.

\begin{lem}\label{lemaprincipal}
Let  $\tau\geq 0$ and functions  $\varphi,\,\psi:[-\tau,\infty)\to(0,1]$  satisfy 
\begin{align*}
  \varphi(t)&=e^{-\int^t_{t-\tau}f(h)\varphi(h)\, dh},\\
  \psi(t)&=e^{-\int^t_{t-\tau}g(h)\psi(h)\, dh}
\end{align*}
for $t\geq  t_g$. Given  $\eta >0$, there exist positive  constants $\varepsilon$  and $\tilde T $ such that: for all $t_0\geq t_g$ and $I\geq 0$, if 
 $$|f(t)-g(t)|<\varepsilon$$ 
for $t\in [t_0,t_0+\tilde T+I]$,   then
\begin{equation}\label{desigualdadlemaprincipal}
  |\varphi(t)-\psi(t)|<\frac{\eta}{4M} 
\end{equation}
for all  $t\in [t_0+\tilde T,t_0+\tilde T+I]$. 
\end{lem}

\begin{proof}
Firstly, note that if  $\tau=0$ then $\varphi(t)=\psi(t)=1$ and there is nothing to prove. Therefore, let us assume that $\tau>0$. Since
$$3M \sqrt{\frac{ t}{     \inf f}  }\left(1-e^{-M\tau}\right)^{t/(2\tau)-1/2} $$
tends to zero as   $t$ tends to infinity, we fix   $\tilde T>0$ such that
$$3M \sqrt{\frac{ t}{     \inf f }} \left(1-e^{-M\tau}\right)^{t/(2\tau)-1/2} <\eta/2,\quad t\geq \tilde T. $$
Moreover, let $\varepsilon>0$ such that 
$$\varepsilon<\frac{ \eta e^{-2M  \tilde T}}{4\tau}.$$
Now let us fix $t_0\geq t_g$, $I\geq 0$, and assume that  
$$|f(t)-g(t)|<\varepsilon$$
for all $t\in [t_0,t_0+\tilde T+I]$. We will show that 
\begin{equation}\label{desigualdadaprobar}
  |f(t_0+\tilde T+a)-g(t_0+\tilde T+a)|<\eta,\quad \textnormal{for all }a\in[0,I]. 
\end{equation}
Let $a\in [0,I]$ and consider $\tilde \psi:[t_0+a-\tau,\infty)\to(0,1]$ the function given by Lemma \ref{hermana} that satisfies both 
$$\tilde \psi(t)=e^{-\int_{t-\tau}^tg(h)\tilde\psi(h)\, dh},$$ 
for all $t>t_0+a$ and
$$|\varphi(t)-\tilde \psi(t)|<2\varepsilon\tau e^{2M(t-t_0-a)}$$
for all $ t\in [t_0+a-\tau,t_0 +\tilde T+I]. $
In particular, the above inequality combined with the definition of $\varepsilon$ implies that
$$|\varphi(t_0+\tilde T+a)-\tilde \psi(t_0+\tilde T+a)|<2\varepsilon\tau e^{2M \tilde T } <\eta/2.$$
Moreover, applying the Lemma \ref{lemacrucial} to  $\psi$ and $\tilde \psi$ we obtain 
\begin{align*}
  |\psi(t)-\tilde \psi(t)|&< 3M \sqrt{\frac{ (t-t_0)}{ \inf f}} \left(1-e^{-M\tau}\right)^{(t-t_0)/(2\tau)-1/2}  <\eta/2
\end{align*}
for $t\geq t_0+\tilde T$. Therefore, it follows that
\begin{align*}
  |\varphi(t_0+\tilde T+a)-\psi(t_0+\tilde T+a)|& \leq |\varphi(t_0+\tilde T+a)-\tilde \psi(t_0+\tilde T+a)|\\
  &\quad +|\tilde\psi(t_0+\tilde T+a)-\psi(t_0+\tilde T+a)|\\
  &<\eta
\end{align*}
and the inequality \eqref{desigualdadaprobar} is proved.
\end{proof}

\subsection{Proof of Theorem \ref{persistenciauniforme}}
 
Set $f(t)=p(z^*(t))$ and fix
\begin{equation}\label{condicionM}
  M=  \max \left\{  p\left(2 \overline \si\right), \frac{1}{4\tau +1} \right\}.
\end{equation} 

\begin{proof} 
Assume that System (\ref{Main System}) is persistent, then Theorem \ref{persistencia} states that there exist positive constants $\eta$ and $T$ such that 
\begin{equation}\label{eq: Hip} 
  \int_{t_1}^{t_2} p(z^*(t-\tau)) \varphi(t-\tau)\, dt> \int_{t_1}^{t_2}\left( D(t)+\eta\right)  \, dt
\end{equation}
for all $t_1> T$, $t_2-t_1> T$ where $\varphi(t)$ is only  dependent on $p(z^*(t))$.
 The remainder of the proof will be divided into four steps with  $\eta$, $T$ and $\varphi(t)$   fixed. 
%Fijemos $\eta$, $T$ y $\varphi$ de aquí al fin de la demostración. Separaremos el resto de la demostración en cuatro steps.

\emph{Step 1.}  Firstly, we  define fundamental constants that do not depend on a particular solution. Since $\si(t)$ is bounded below by a positive constant and $z^*(0)>0$ we deduce that $p(z^*(t))$ is bounded below by a positive constant. As we already have a given $\eta$, let    $\tilde T$ and   $\varepsilon$ be  the constants given by Lemma \ref{lemaprincipal}. Note  that the functions $f(t)$ and $\varphi(t)$ are both fixed. Moreover, it is possible to consider $\varepsilon$ small enough so that 
\begin{equation}\label{segundadesigualdad}
  \varepsilon\leq \eta/4.
\end{equation}
To conclude the first step  we define 
$$
L:=\max_{\xi\in [0,2\overline{s^{(0)}}]} p'(\xi),
$$
and
%lets consider $L>0$ an upper bound for $p'(\xi)$ and $\xi\leq 2 \overline\si $, finally define  
 $$\delta :=\frac{\varepsilon e^{- \overline D(\tilde T+T+2\tau)}}{2L(1+M\tau)} .$$  
 
\emph{Step 2.} Now let $R$ and $\alpha$ be positive constants, we recall the statement of the Theorem \ref{persistenciauniforme}: there is $T^{in}(R,\alpha)\geq 0$ such that for a given $
x(t)\geq \delta$  for all   $t\geq T^{in}.$
Therefore consider $t_0\geq  \tau$ such that 
\begin{equation}\label{definiciont0}
  \left(\overline \si+2R+Rp(R)\tau\right)e^{-\int_0^{t_0- \tau}D(r)\, dr}\leq \min\left\{\frac{\varepsilon}{2L},\overline \si\right\}
\end{equation}
and define
\begin{align*}
  I^{in}&=\max\left\{T,\frac{2}{\eta}\ln\left(\frac{\varepsilon e^{\overline D(t_0+\tilde T)}}{2L\alpha(1+M\tau) }\right)\right\},\\
  T^{in}&=t_0+\tilde T+I^{in}+\tau.
\end{align*}
Given $(s(t),x(t))$ a particular solution of the System (\ref{Main System}) such that $\|(s^{in},x^{in})\|\leq R$ and $x(\tau)\geq \alpha$ we firstly  prove the existence of $T_*^{in}\in [t_0- \tau,T^{in}]$ such that
\begin{equation}\label{testrella}
x(T^{in}_*)\geq \frac{\varepsilon}{2L(1+M\tau)}. 
\end{equation}
%is satisfied for some $T^{in}_*\in[t_0-2\tau,t_0+\tilde T+I^{in}]$.
To obtain a contradiction, we suppose the opposite. In particular,
\begin{equation}\label{xchica}
x(t)<\frac{\varepsilon}{2L(1+M\tau)}   
\end{equation}
for all $t\in [t_0- \tau,T^{in}-\tau]$.

Firstly,  note that for $t\geq 0$ it follows 
\begin{equation}\label{cotay} 
  y(t) \leq \int_{t-\tau}^t x(h)p(s(h))   \,dh\leq   \int_{t-\tau}^t \|x_t\|p(\|s_t\|)  \,dh\leq  \|x_t\|p(\|s_t\|)\tau 
\end{equation} where we use the notation introduced in (\ref{definicionenC}).

Secondly, using   Lemma \ref{funcionz} and the  previous inequality   evaluated in $t=0$, combined with the properties of the initial data, it turns out that 
\begin{equation}\label{otracota}
\begin{aligned}
  |(z^*-s-x-y)(t)|&\leq |(z^*-s-x-y)(0)|e^{-\int_0^{t }D(r)\,dr}\\
  &\leq \left(z^*(0) +||s^{in}||+||x^{in}||+y(0)\right)e^{-\int_0^{t }D(r)\,dr}\\
  &\leq \left(\overline s+2R+Rp(R)\tau\right) e^{-\int_0^{t }D(r)\,dr}.
\end{aligned}
\end{equation}
Now, from the solution $(s,x)(t)$ we define $g(t):=p(s(t))$ and consider the function  $\psi(t)=\psi(s,x)(t)$ given by (\ref{eq: 4.2}). 
We claim that if $t_0=t_g$ then $p(s(t))$ verifies (\ref{condiciong}). Indeed, let us consider $t\geq t_0-\tau$. Then,    using that $z^*(t)\leq \overline s$ for all $t$  (recall Definition (\ref{z*})), and Inequalities (\ref{definiciont0}) and (\ref{otracota})  it follows that
\begin{align*}
  s(t)&\leq (s+x+y)(t)\leq |(z^*-s-x-y)(t)|+z^*(t)\leq 2\, \overline s,
\end{align*}
and (\ref{condiciong}) holds by the definition of $M$ in (\ref{condicionM}).
We remark that we construct $p(s(t))$ satisfying (\ref{condiciong}) for the purpose of applying  Lemma  \ref{lemaprincipal}. 

Next, combining (\ref{cotay}) with   (\ref{condiciong}) and (\ref{xchica})   we obtain
\begin{align*}
  y(t)&\leq    \frac{M \tau\varepsilon}{2L(1+M\tau)}
\end{align*}
for all   $t\in[t_0,T^{in}-\tau]$.

Now, using this last inequality together with (\ref{definiciont0}), (\ref{xchica}) and (\ref{otracota}), it turns out that 
\begin{align*}
  |(z^*-s)(t)|&\leq |(z^*-s-x-y)(t)|+(x+y)(t)\\
  &\leq \left(\overline s+2R+Rp(R)\tau\right) e^{-\int_0^{t }D(r)dr}+\frac{ \varepsilon}{2L(1+M\tau)}+ \frac{M \tau\varepsilon}{2L(1+M\tau)}\\
  &\leq \frac{\varepsilon}{2L}+\frac{ \varepsilon}{2L(1+M\tau)}+ \frac{M \tau\varepsilon}{2L(1+M\tau)}\\
  &= \varepsilon/L , 
\end{align*}
for  all $t\in[t_0,T^{in}-\tau]$. Then, by the mean value Theorem and the definition of $L$, it follows that
$$|p(z^*(t))-p(s(t))|\leq L|(z^*-s)(t)|<\varepsilon$$
for   the same interval of time. Recall that $T^{in}-\tau=t_0+\tilde T+I^{in}$ and apply Lemma  \ref{lemaprincipal} with $I=I^{in}$. This enables to deduce  that \eqref{desigualdadlemaprincipal} holds, which together with \eqref{segundadesigualdad}  implies
\begin{align*}
  |p(z^*(t))\varphi(t)-p(s(t))\psi(t)|&\leq |p(z^*(t))\varphi(t)-p(z^*(t))\psi(t)| \\
  &\quad +|p(z^*(t))\psi(t)-p(s(t))\psi(t)|\\
  &\leq p(z^*(t))|\varphi(t)-\psi(t)| +|p(z^*(t))-p(s(t))|\psi(t)\\
  &\leq M \frac{\eta}{4M}+\frac{\eta}{4 }\\
  &=\eta/2,
\end{align*}
and then
\begin{equation}\label{eq: 4.11}
 p(z^*(t))\varphi(t)-\eta/2\leq p(s(t))\psi(t)   
\end{equation}
for $t\in[t_0 +\tilde T,T^{in}-\tau]$.

%\begin{align*}
%  |p(z^*(t-\tau))\varphi(t)-p(s(t-\tau))\psi(t)|&\leq |p(z^*(t-\tau))\varphi(t)-p(z^*(t-\tau))\psi(t)|\\
%  &\quad +|p(z^*(t-\tau))\psi(t)-p(s(t-\tau))\psi(t)|\\
%  &\leq p(z^*(t-\tau))|\varphi(t)-\psi(t)|\\
%  &\quad +|p(z^*(t-\tau))-p(s(t-\tau))|\psi(t)\\
%  &\leq M \frac{\eta}{4M}+\frac{\eta}{4 }\\
%  &=\eta/2
%\end{align*}
Since $t_0+\tilde T-\tau=T^{in}-I^{in}$, and from (\ref{eq: Hip}) and (\ref{eq: 4.11})  
%we obtain  
%$$\int_{t_0 +\tilde T}^{t_0 +\tilde T+I^{in}}p(z(h-\tau))\varphi(h-\tau)-D(h)-\eta\, dh\geq 0, $$
 we deduce that  
\begin{equation}\label{eq: 4.12}
    \begin{aligned}
    &\int_{T^{in}-I^{in}}^{T^{in} }\left(p(s(h-\tau))\psi(h-\tau)-D(h) \right)dh   \\
&\hspace{1cm} \geq   \int_{T^{in}-I^{in}}^{T^{in} }\left(p(z^*(h-\tau))\varphi(h-\tau)- \eta/2 -D(h)\right) \, dh\\
   &\hspace{1cm} \geq   \int_{T^{in-I^{in}}}^{T^{in} } \eta/2   \, dh\\
 &\hspace{1cm}   =\displaystyle I^{in}\eta/2   . 
    \end{aligned}
\end{equation}

%\begin{align*}
%  &\int_{t_0 +\tilde T+\tau}^{t_0 +\tilde T+I^{in}+\tau}\left(p(s(h-\tau))\psi(h-\tau)-D(h) \right)\\
%  &\quad \geq \int_{t_0+\tilde T+\tau}^{t_0+\tilde T+I^{in}+\tau}\left(p(z^*(h-\tau))\varphi(h-\tau)- \eta/2 -D(h)\right) \, dh\\
%  &\quad \geq \int_{t_0+\tilde T+\tau}^{t_0+\tilde T+I^{in}+\tau} \eta/2   \, dh\\
%  &\quad =I^{in}\eta/2.
%\end{align*}
 From the second equation of the System (\ref{Main System}) we obtain the following inequality
\begin{equation}\label{cotainferiorderivada}
 \frac{d}{dt}x(t)\geq  -\overline Dx(t),
\end{equation}
which combined with  the assumption $x(\tau)\geq \alpha$ gives 
$$x(T^{in}-I^{in})=x(t_0+\tilde T+\tau)\geq e^{-\overline D(t_0+\tilde T )}x(\tau)\geq e^{-\overline D(t_0+\tilde T )}\alpha.$$
Now by using      (\ref{eq: 3.7}),  the above inequality and (\ref{eq: 4.12}), together with   
 the definition of $I^{in}$ we have
\begin{align*}
  x(T^{in} )&=x(T^{in}-I^{in})e^{\int_{T^{in}-I^{in}}^{T^{in}}\left(p(s(h-\tau))\psi(h-\tau)-D(h)\right)\, dh} \\
  &\geq e^{-\overline D(t_0+\tilde T)}\alpha e^{I^{in}\eta/2}\\
  &\geq \frac{\varepsilon}{2L(1+M\tau)}, 
\end{align*}
which contradicts \eqref{xchica} and, therefore, there exists  $T^{in} _*\leq T^{in}$ that satisfies \eqref{testrella}.

\emph{Step 3.} It remains to see that $ x(t)\geq \delta $ for all $t\geq T^{in}$. We will prove that the above holds for  $t\geq T^{in}_*$. To do this, define 
$$\mathcal{S}=\left\{t\geq T^{in}_*:x(h)\geq \delta \text{ for all }h\in[T^{in}_*,t]\right\}.$$
We claim that 
\begin{equation}\label{primerapertenencia}
  T^{in}_*+\tilde T+T+2\tau\in\mathcal{S}.
\end{equation}  
Let $h\in[T^{in}_*,T^{in}_*+\tilde T+T+2\tau]$, 
 as before  (\ref{eq: 3.7}),  (\ref{cotainferiorderivada}), (\ref{eq: 4.12}), combined with the properties of $T^{in}_*$ and the definition of $\delta$, implies
 $$x(h)\geq x(T^{in}_*)e^{-\int_{T^{in}_*}^{h}D(r)\, dr}\geq \frac{\varepsilon e^{-\overline D(h-T^{in}_*)}}{2L(1+M\tau)} \geq \frac{\varepsilon e^{-\overline D( \tilde T+T+2\tau)}}{2L(1+M\tau)} \geq \delta $$
and (\ref{primerapertenencia}) holds.

Next we define $t^*:=\sup\mathcal{S}$ and claim that  $t^*=\infty$. Suppose, contrary to our claim, that  $T^{in}_*+\tilde T+T+2\tau<t^*<\infty.$ By continuity of $x(t)$ it  follows that
 \begin{equation}\label{nuevoabsurdo}
  x(t^*)=\delta .
\end{equation} 
Again by \eqref{cotainferiorderivada},  for all  $t\in[t^*-\tilde T -T-2\tau,t^*]$ we have that 
$$x(t^*)\geq x(t)e^{-\int_{t}^{t^*}D(r)\, dr}\geq x(t) e^{-\overline D(\tilde  T +T+2\tau)},$$
which implies
$$x(t)\leq \delta e^{\overline D(\tilde T+T+2\tau)}=\frac{\varepsilon}{2L(1+M\tau)} $$
  due to the  definition of $\delta$. We repeat the reasoning done in Step 2, but now we apply the Lemma \ref{lemaprincipal} considering $I=T$. Note that
\begin{align*}
  y(t)& \leq  \frac{M \tau\varepsilon}{2L(1+M\tau)}
\end{align*}
%for   $t\in [t^*-\tilde T-T-\tau,t^*] $ 
for $ t\in [t^*-\tilde T-T-\tau,t^*]$, therefore we can deduce that
$$|p(z^*(t-\tau))\varphi(t-\tau)-p(s(t-\tau))\psi(t)| \leq  \eta/2$$
for $t\in[t^* -T ,t^*]$, and  consequently
$$\int_{t^* -T}^{t^*}\left(p(s(h-\tau))\psi(h-\tau)-D(h) \right)\, dh \geq  T\eta/2.$$
Finally, combining (\ref{eq: 3.7}), the previous inequality, and the fact that $t^*-T\in \mathcal{S}$, it follows that
$$x(t^*)= x(t^*-T)e^{\int_{t^*-T}^{t^*}\left[p(s(h-\tau))\psi(h-\tau)-D(h)\right]dh}\geq x(t^*-T)e^{T\eta/2}>x(t^*-T)\geq \delta$$ 
%because $t^*-T\in \mathcal{S}$. 
Hence we have a contradiction with \eqref{nuevoabsurdo},  so $t^*$ is infinite and the theorem is proved.
\end{proof}

\section{Proofs for periodic case}\label{pruebasperiodico}

\subsection{Ideas for the proofs}
By Theorem \ref{persistenciaperiodica} if
$$ \langle p(z^*) \varphi\rangle >\langle D\rangle$$
then System (\ref{Main System}) is persistent. Therefore, Theorem \ref{teoremadelaextincion} is a converse criteria, namely, that every solution tends to the extinction  if and only if
\begin{equation}\label{extincion}
  \langle p(z^*)\varphi\rangle \leq \langle D\rangle.
\end{equation}
Thus this inequality is a threshold for the vanishing of biomass. We prove Theorem \ref{teoremadelaextincion} by contradiction. If $$\limsup_{t\to\infty} x(t)>0$$  then $z^*$ is larger than $s$ from a certain moment. By Lemma \ref{vuelta} $\int p(z^*)\varphi$ is larger than $\int p(s)\psi$ (except for a constant). Therefore, if 
$$\langle p(z^*)\varphi\rangle< \langle D\rangle$$
then 
$$e^{\int (p(z^*(t-\tau))\varphi(t-\tau)-D(t))dt}$$
goes to zero, which implies that
$$e^{\int (p(s(t-\tau))\varphi(t-\tau)-D(t))dt}$$
goes to zero and using (again) Equality (\ref{eq: 3.7}) we get a contradiction. In the case that 
$$\langle p(z^*)\varphi\rangle= \langle D\rangle$$
we use  ideas  from the prove of \cite[Theorem 1]{amster2020dynamics} to show that 
$$\int (p(z^*)\varphi-p(s)\psi)$$
tends to infinity and repeat the previous idea to get a contradiction.

 As mentioned, the goal of Theorem \ref{teoremadeexistencia} is to use the Horn's fixed point Theorem  \cite[Theorem 6]{horn1970some} combined with Theorem \ref{persistenciauniforme}. Therefore,  we define three convex sets $S_0\subset S_1\subset S_2\subset \mathcal{C}$, with $S_0$ and $S_2$ compact, and $S_1$ open (relative to $S_2$).   By inspiration in the proof of \cite[Theorem 1]{teng2011periodic} we use  Arzelà–Ascoli Theorem to prove the compactness of $S_0$ and $S_2$. Therefore, applying Horn's Theorem to the Poincaré operator we get the existence of a positive $\omega$-periodic solution if the System (\ref{Main System}) is persistent. We emphasize that  we   give an explicit representation of $S_1$ as the intersection of $S_2$ with an open set. This ensures that $S_1$ is open relative to $S_2$ and that the hypothesis of the Horn's Theorem are satisfied.

% We make an observation about the proof of \cite[Theorem 1]{teng2011periodic}. We can not get the reason why the set $S_t^{(1)}$ defined there is open (or open relative to $S_t^{(2)}$). However, this is not a problem in the present article since we   give an explicit representation of $S_1$ as the intersection of $S_2$ with an open set. We have written to the corresponding author, but at the time of this publication we have not received a response.\todo{Yo creo que no debemos incluir esta frase.}

Finally, to introduce the proof of Theorem \ref{teoremadeestabilidad} first consider the case without delay ($\tau=0$). Fix $(s_1,x_1)$ a solution with not null initial condition and $(s_2,x_2)$ a positive $\omega$-periodic solution. Note that $s_1+x_1\approx z^*$ for large times and $s_2+x_2= z^*$ which implies that
$x_1-x_2\approx -(s_1-s_2)$
and then
\begin{align*}
    \frac{d}{dt}(s_1-s_2)&=-D(s_1-s_2)-x_1p(s_1)+x_2p(s_2)\\
    &=-D(s_1-s_2)-x_1(p(s_1)-p(s_2))-(x_1-x_2)p(s_2)\\
    &=-D(s_1-s_2)-x_1p'(\zeta)(s_1-s_2)-(x_1-x_2)p(s_2)\\
    &\approx -(s_1-s_2)(D+x_1p'(\zeta)-p(s_2))\\
\end{align*}
where, for simplicity, we do not write the dependence of time. Observe that $\langle D\rangle =\langle p(s_2)\rangle$ since $x_2$ is $\omega$-periodic (recall (\ref{eq: 3.7}) and note that $\psi=1$). Then, if the system is persistent, $x_1$ is bounded below by a positive constant (from a certain time) and $s_1-s_2$ tends to zero when time goes to infinity. This shows that $(s_2,x_2)$ is attractive. To apply this idea to the general case ($\tau\geq 0$) and inspired in a generalized Gronwall-type inequality given in \cite[Lemma 2.4]{choi2017cucker}, the goal is to define an appropriated function $w$ that bounds $s_1-s_2$ and tends to zero when the time tends to infinity.

\subsection{Auxiliary lemmas for the periodic chemostat}
We need the following Lemma inspired by  \cite[Lemma 4.3]{cartabia2022persistence}.  
\begin{lem}\label{vuelta}
Let $t_0\geq 0$ and $\tau>0$. Consider $\varphi,\,\psi:[t_0-\tau,\infty)\to(0,1]$ satisfying 
\begin{align*}
    \varphi(t)&=e^{-\int^t_{t-\tau}f(h)\varphi(h)\, dh}, \\
    \psi(t)&=e^{-\int^t_{t-\tau}g(h)\psi(h)\, dh}  
\end{align*}
for all $t> t_0$ and assume that $ f(t)\geq  g(t)$ for all  $t\geq t_0$. Then    
$$\int_{t_1 }^{t_2}f(h-\tau)\varphi(h-\tau)\, dh+\tau M\geq  \int_{t_1 }^{t_2}  g(h-\tau)\psi(h-\tau) \, dh$$
for all  $t_2 \geq t_1\geq t_0$.
\end{lem}

\begin{proof} 
The proof will be divided into three steps.
 
\emph{Step 1.} Observe that if  $ \varphi(  t )\leq  \psi(  t ) $ for $t\geq t_0$, then we get   
\begin{align*}
  e^{\int_{t -\tau}^{t} \left(f(h)\varphi(h)-g(h)\psi(h)\right) \, dh}&=\frac{\psi(t)}{\varphi(t)}\geq  1 
\end{align*}
which implies
$$\int_{t-\tau}^{t} \left(f(h)\varphi(h)-g(h)\psi(h) \right)\, dh\geq 0.$$

On the other hand, if there are $h_2\geq h_1\geq t_0$ such that  $  \varphi(  t )\geq \psi(  t ) $ for all $t\in [h_1,h_2]$ then  
\begin{equation*} 
   \int_{h_1}^{h_2} \left(f(h)\varphi(h)-g(h)\psi(h) \right)\, dh \geq 0, 
  \end{equation*}
since 
$$f(t)\varphi(t)-g(t)\psi(t)\geq g(t)\varphi(t)-g(t)\varphi(t)=0.$$

\emph{Step 2.} We claim that there is a finite decreasing sequence $\{h_n\}_{1\leq n\leq N}\subset\R$ with the properties that $ \varphi( h_n )\leq  \psi(  h_n)$ for all $1\leq n\leq N-1$ and that $ \varphi( t )\geq  \psi(  t)$ if  
$$t\in I=\bigcup_{n=1}^{N-1}[h_{n+1},h_n-\tau]\cup [h_1,t_2-\tau] .$$
  Indeed, define  
\begin{equation*}
  h_1 =\left\{\begin{aligned}
    &t_2-\tau &\text{if }    \varphi( t_2-\tau )< \psi(  t_2-\tau),\\
   &\inf\left\{t\geq   t_1-\tau:   \varphi(  h )\geq \psi(  h ) \text{ for all } h\in[t,t_2-\tau]\right\}& \text{otherwise.}
  \end{aligned}\right.
\end{equation*}
For $n\geq 1$ and while $t_1\leq h_{n }-\tau$, define
\begin{equation*}
  h_{n+1} =\left\{\begin{aligned}
    &h_{n }-\tau& \text{if }  \varphi( h_{n }-\tau)< \psi( h_{n }-\tau),\\
   &\inf\left\{\begin{array}{l}
    t\geq   t_1-\tau :   \varphi(  h )\geq \psi(  h )\\
   \text{ for all }h\in[t,h_n-\tau]
   \end{array}
   \right\} &\text{otherwise.}
  \end{aligned}\right.
\end{equation*}
 Observe that the sequence ends when $t_1> h_{N }\geq t_1-\tau$ satisfying the statement of the claim. Moreover, note that as consequence of the definition of the  decreasing sequence $\{h_n\}_{1\leq n\leq N}$  and the Step 1 we obtain the following inequalities
\begin{equation}\label{caminoenreverso}
  \begin{array}{rcl}
  \int_{h_n-\tau }^{h_n}f(h)\varphi(h)\, dh&\geq&  \int_{h_n-\tau }^{h_n} g(h)\psi(h) \, dh, \textnormal{ for all }1\leq n\leq N-1,\\ \\
  \int_If(h)\varphi(h)\, dh&\geq & \int_I  g(h)\psi(h) \, dh. 
  \end{array}
\end{equation} 
Whence  we obtain
$$\int_{h_N }^{t_2-\tau}f(h)\varphi(h)\, dh\geq  \int_{h_N }^{t_2-\tau}  g(h)\psi(h) \, dh.$$
 
\emph{Step 3.} Finally, observe that 
\begin{align*}
  \int_{t_1-\tau}^{t_2-\tau}f(h)\varphi(h)\, dh&\geq \int_{h_N}^{t_2-\tau}f(h)\varphi(h)\, dh\\
  &\geq \int_{h_N -\tau}^{t_2-\tau}  g(h)\psi(h) \, dh \\
  &= \int_{t_1 -\tau}^{t_2-\tau}  g(h)\psi(h) \, dh -\int_{t_1-\tau }^{h_N-\tau}  g(h)\psi(h) \, dh \\
  &\geq \int_{t_1-\tau }^{t_2-\tau} g(h)\psi(h) \, dh-\tau M,  
\end{align*}
  and the lemma is proved.
\end{proof}

The following lemma is fundamental to the proof of the Theorem \ref{teoremadeestabilidad}. Applied to a periodic solution of the System (\ref{Main System}), it states that  $(x+y)(t-\tau)$ diluted over $\tau$ units of times is exactly $x(t)$. 

\begin{lem}\label{lemamuyinteresante}
Every solution $(s,x)(t)$ of the System 
(\ref{Main System}) with not null initial condition, and the corresponding function $y(t)$ defined in (\ref{funciony}) satisfy
$$ x(t) - (x+y)(t-\tau)e^{-\int_{t-\tau}^tD(r)dr}  =\left(x(\tau)e^{\int_0^\tau D(r)\,dr}-(x+y)(0)\right)e^{-\int_{0}^tD(r)dr} $$
for all $t\geq \tau.$ 
In the particular case where $(s,x)(t)$ is an $\omega$-periodic (non-trivial) solution it follows that 
$$ x(t) = (x+y)(t-\tau)e^{-\int_{t-\tau}^tD(r)dr} .$$
\end{lem}

\begin{proof}
 Firstly, note that for a given solution $(s,x)(t)$ with not null initial condition  it follows that 
$$\frac{d}{dt}(x+y)(t-\tau)=-D(t-\tau)(x+y)(t-\tau)+x(t-\tau)p(s(t-\tau))$$
and consequently
\begin{align*}
    \frac{d}{dt}\left((x+y)(t-\tau)e^{-\int_{t-\tau}^tD(r)\,dr}\right)&= x(t-\tau)p(s(t-\tau)) e^{-\int_{t-\tau}^tD(r)\,dr}\\
    &\quad - D(t)(x+y)(t-\tau) e^{-\int_{t-\tau}^tD(r)\,dr}.
\end{align*}
Whence we deduce that
\begin{align*}
  &\frac{d}{dt}\left(x(t) -(x+y)(t-\tau)e^{-\int_{t-\tau}^tD(r)\,dr}  \right)=\\
  &\hspace{5.5cm} -D(t)\left(x(t) -(x+y)(t-\tau)e^{-\int_{t-\tau}^t D(r)\, dr} \right)
\end{align*}
 which provides the first part of the lemma.
 
Note that 
\begin{equation}\label{eq: 5.3}
\left| x(t) -(x+y)(t-\tau) e^{-\int_{t-\tau}^tD(r)\,dr} \right|\to 0 
\end{equation}
as $t$ tends to infinity. In particular, if $(s,x)(t)$ is a non-trivial periodic solution of (\ref{Main System}), the right side of (\ref{eq: 5.3}) is also an $\omega$-periodic function  and converges to zero as $t$ tends to infinity, therefore it is identically zero, and the proof is complete.
\end{proof}

\subsection{Proof of Theorem \ref{teoremadelaextincion}}

\begin{proof}

Assume that System (\ref{Main System}) is not persistent, given a solution $(s,x)(t)$ we will prove that $x(t)$ towards to zero as $t\to\infty$.    If $\lim_{t\to\infty}x(t)=0$ there is nothing to prove. Otherwise there exists $\varepsilon>0$ such that   
\begin{equation}\label{existeepsilon}
    \limsup_{t\to\infty}x(t)>\varepsilon
\end{equation} 
and we look for a contradiction. By using Lemma \ref{funcionz}, it follows that
$$|(z^*-s-x-y)(t)|\to 0, \quad t\to \infty.$$
Moreover, since the functions $z^*(t),\,s(t),\,x(t)$ and $y(t)$ are positive and 
$$|(z^*-s-x-y)(t)|<\varepsilon$$
from a certain time, by (\ref{existeepsilon}) we conclude that there exists  $t_0\geq 0$  such that $z^*(t_0)\geq s(t_0)$. Furthermore, $z^*(t)
- s(t)$ satisfy
\begin{align*}
  \frac{d}{dt}(z^*-s)(t)&=-D(t)(z^*-s)(t)+x(t)p(s(t))\geq -D(t)(z^*-s)(t)
\end{align*}
thus, from the above differential inequality, we see that $z^*(t)\geq s(t)$ for $t\geq t_0$. 
Note that, by assumption that System (\ref{Main System}) is not persistent, one of the following identities holds
\begin{equation}\label{desigualdad}
  \langle D\rangle >\langle p(z^*)\varphi\rangle,
\end{equation}
or
\begin{equation}\label{igualdad}
  \langle D\rangle =\langle p(z^*) \varphi\rangle.
\end{equation}
The rest of the proof fall naturally into two cases.

\emph{Case 1.}  Assume that (\ref{desigualdad}) holds. Consider again the functions $f(t)=p(z^*(t))$ and $g(t)=p(s(t))$, then by (\ref{eq: 3.7}) and Lemma \ref{vuelta} we obtain for $t\geq t_0$ 
\begin{align*}
x(t)&=x(t_0)e^{\int_{t_0}^t\left(p(s(r-\tau))\psi(r-\tau)-D(r)\right)\, dr}\\
&=x(t_0)e^{\int_{t_0}^t\left(p(s(r-\tau))\psi(r-\tau)-p(z^*(r-\tau)\varphi(r-\tau)\right)\, dr +\int_{t_0}^t\left(p(z^*(r-\tau))\varphi(r-\tau)-D(r)\right)\, dr}\\
  &\leq x(t_0) e^{\tau M}e^{- \int_{t_0}^t\left(D(r)-p(z^*(r-\tau))\varphi(r-\tau)\right)\, dr  }.
\end{align*}
From (\ref{desigualdad}) it may be conclude that $x(t)\to 0$ as $t\to \infty$ and we get a contradiction.

\emph{Case 2.} Assume now that (\ref{igualdad}) holds. We claim that   
\begin{equation}\label{alinfinitoymasalla}
  \int_{t_0}^t\left(p(z^*(r-\tau))\varphi(r-\tau)-p(s(r-\tau))\psi(r-\tau)\right)\,dr\to\infty
\end{equation}
as $t\to\infty$. Indeed, otherwise we could find $K>0$ such that: for all $t>t_0$  there exists $t_1>t$ such that 
$$\int_{t_0-\tau}^{t_1-\tau}\left(p(z^*(r))\varphi(r)-p(s(r))\psi(r)\right)\,dh<K.$$
Now combining the above  inequality together with the Lemma \ref{vuelta} we have that for all $t>t_0$ it follows 
\begin{align*}
\int_{t_0-\tau}^{t-\tau}\left(p(z^*(r))\varphi(r)-p(s(r))\psi(r)\right)\,dr&=\int_{t_0-\tau}^{t_1-\tau}\left(p(z^*(r))\varphi(r)-p(s(r))\psi(r)\right)\,dr\\
  & -\int_{t-\tau}^{t_1-\tau} \left(p(z^*(r))\varphi(r)-p(s(r))\psi(r)\right)\,dr\\
  &\leq K+\tau M.
\end{align*}
 Consequently,  from (\ref{eq: 3.7}) we obtain for all $t\geq t_0$ 
 \begin{align*}
\ln(x(t))&=\ln(x(t_0))+\int_{t_0}^{t}\left(p(s(r-\tau))\psi(r-\tau)-D(r)\right)\, dr\\
   &=\ln(x(t_0))+\int_{t_0}^{t}\left(p(s(r-\tau))p(\psi(r-\tau)-p(z^*(r-\tau))\varphi(r-\tau)\right)\, dr\\
   &\quad +\int_{t_0}^{t}\left(p(z^*(r-\tau))\varphi(r-\tau)-D(r)\right)\, dr\\
   &\geq \ln(x(t_0))-K-\tau M+\min_{l\in[0,\omega]} \int_{ 0}^{l}\left(p(z^*(r-\tau))\varphi(r-\tau)-D(r)\right)\, dr,
 \end{align*}
which implies that $\ln(x(t))$ is bounded from below for all $t\geq t_0$, i.e., the solution is persistent. This contradicts our assumption   and  \eqref{alinfinitoymasalla} is proved.

Now, by using again (\ref{eq: 3.7}) for $t\geq t_0$, we obtain  
\begin{align*}
x(t)&=x(t_0)e^{\int_{t_0}^tp(s(r-\tau))\psi(r-\tau)-D(r)\, dr}\\
 &=x(t_0)e^{\int_{t_0}^tp(s(r-\tau))\psi(r-\tau)-p(z^*(r-\tau)\varphi(r)\, dr +\int_{t_0}^tp(z^*(r-\tau))\varphi(r-\tau)-D(r)\, dr}.
\end{align*}
Finally, notice that (\ref{igualdad}) implies that for all $t\geq t_0$ we have
$$ \int_{t_0}^tp(z^*(r-\tau))\varphi(r-\tau)-D(r)\, dr \leq \max_{l\in[0,\omega]} \int_{0}^{l}p(z^*(r-\tau))\varphi(r-\tau)-D(r)\, dr, $$ consequently we obtain
\begin{align*}
x(t)  &\leq x(t_0)M e^{-\int_{t_0}^tp(z^*(r-\tau))\varphi(r-\tau)-p(s(r-\tau)\psi(r-\tau)\, dr  },
\end{align*}
where $M:=\max_{l\in[0,\omega]}\left\{e^{ \int_{0}^lp(z^*(r-\tau))\varphi(r-\tau)-D(r)\, dr}\right\}$. From (\ref{alinfinitoymasalla}) we conclude that $x(t)\to 0$ as $t\to \infty$, which contradicts (\ref{existeepsilon}) and the proof is complete.
\end{proof}

\subsection{Proof of Theorem \ref{teoremadeexistencia}}

\begin{proof}
The proof will be divided into two steps, we recall  that $ \mathcal{C}:=C([-\tau,0]\to\R^2)$.

\emph{Step 1.} Let $\delta>0$ the value given by Theorem \ref{persistenciauniforme} and consider 
$$T^{in}=T^{in}(R,\alpha)\quad \textnormal{with }  R=3\, \overline s\quad \text{ and }\alpha= e^{-\overline D\tau}\delta/2.$$
We now claim  that any solution of System  
(\ref{Main System}) with initial condition 
$$\|(s^{in},x^{in})\|\leq 3\,  \overline s $$
satisfies 
\begin{equation}\label{cotaR0}
  \|(s ,x )_t\|\leq R_0 \textnormal{ for all }t\geq 0, 
\end{equation}
where
$$R_0= z^*(0)+6\, \overline s +3\, \overline s p\left(3\,\overline s \right)\tau +\overline s.$$
To show this, first note that, by using the arguments to obtain  (\ref{cotay})  in the proof of Theorem \ref{persistenciauniforme}, it follows that
$$|y(0)|\leq     \|(s^{in},x^{in})\|p(\|(s^{in},x^{in})\|)\tau\leq 3\, \overline s p(3\, \overline s )\tau.$$
By definition of $R_0$ and applying the Lemma \ref{funcionz} again, we obtain that
\begin{equation}\label{cotas+x}
\begin{aligned}
  |(s,x)(t)|&\leq (s+x)(t)\\
  &\leq (s+x+y)(t)\\ 
  &\leq |(z^*-s-x-y)(t)|+|z^*(t)|\\
  &\leq |(z^*-s-x-y)(0)|e^{-\int_0^tD}+|z^*(t)|\\
  &\leq \max\{z^*(0),(s+x+y)(0)\}+\overline s\\
  &\leq  z^*(0)+6\, \overline s +3\, \overline s p\left(3\, \overline s \right)\tau +\overline s\\
  &=R_0 
\end{aligned} 
\end{equation}
for all   $t\geq -\tau$.

\emph{Step 2.} Let us define 
\begin{align*}
  %H&=\max\left\{|f(t,\phi)|:(t,\phi)\in \R^+_0\times C,\, \|\phi\|\leq R_0\right\},\\
  S&=\left\{\phi\in \mathcal{C}: |\phi(h)-\phi(r)|\leq \sqrt{2} R_0\left(\overline D+p(R_0)\right)|h-r|\text{ for all }h,\,r\in [-\tau,0]\right\}\\
  S_0&=\left\{\phi\in S:\|\phi\|\leq 2\, \overline s,\,|\phi_2(0)|\geq \delta \right\},\\
  S_1&=\left\{\phi\in S:\|\phi\|< 3\, \overline s,\,|\phi_2(0)|> \delta/2 \right\},\\
  S_2&=\left\{\phi\in S:\|\phi\|\leq R_0,\,|\phi_2(0)|\geq  e^{-\overline DT^{in}}\delta/2 \right\}.\\
\end{align*}
Note that $S_0\subset S_1\subset S_2$ are convex subset of $\mathcal{C}$,   $S_0$ and $S_2$ are compacts as consequence of Arzelà–Ascoli Theorem  and  
$$S_1=S_2\cap \left\{\phi\in \mathcal{C}:\|\phi \|<3\, \overline s ,\,|\phi_2(0)|>\delta/2\right\} $$
consequently, $S_1$ is open relative to $S_2$.

Next we define the Poincar\'e operator given by
$$\begin{array} {rccl}
  P:&S_2&\to& \quad \mathcal{C}  \\
  &\phi&\mapsto &(s,x) (\cdot+\omega,\phi).
\end{array}$$
where $(s,x)(t,\phi)$ is the solution of System (\ref{Main System}) at time $t$, with initial condition  $\phi$ (recall (\ref{notacioncondicionesiniciales})). Notice that for a given $k\in \N\cup\{0\}$,  $P^k(\phi)$ is the solution $(s,x)(t,\phi)$ over the time interval $[k\omega-\tau,k\omega]$ (including when $k\omega-\tau<0$).

To use Horn's  fixed point Theorem we need to prove that
$P^k(S_1)\subseteq S_2$ for all $k\in \N$. Let $\phi\in S_1$ and $(s,x)$ the solution corresponding to the initial condition $\phi$, it suffices to show that $(s,x)_t\in S_2$ for $t\geq 0$. From Equation (\ref{cotaR0}) we have $|(s,x)_t|\leq R_0$ while, for $t\leq T^{in}$  inequality (\ref{eq: 4.12}) implies %and, as before, obtain  
%$$\frac{d}{dt}x(t)\geq -\overline Dx(t)$$
%therefore, if $t\leq T^{in}$, it follows that 
$$|x(t)|\geq e^{-\overline Dt}|x(0)|\geq e^{-\overline DT^{in}}\delta/2.$$
Hence
\begin{equation}\label{cotaentau}
  x(\tau)\geq e^{-\overline D\tau}\delta/2
\end{equation}
thus, by using Theorem \ref{persistenciauniforme}, if $t\geq T^{in}$, then
$$|x(t)|\geq \delta\geq e^{-\overline DT^{in}}\delta/2.$$ 
Now, it turns out that
\begin{align*}
  |(s,x)(h)-(s,x)(r)|&=\sqrt{(s(h)-s(r))^2+(x(h)-x(r))^2}\\
  &=\sqrt{s^{'2}(\xi_s)(h-r)^2+x^{'2}(\xi_x)(h-r)^2}\\
  &=|h-r|\left(\left(D(\xi_s)(\si(\xi_s)-s(\xi_s))-x(\xi_s)p(s(\xi_s))\right)^2\right.\\
  &\quad \left.+\left(-D(\xi_x)x(\xi_x)+x(\xi_x-\tau)p(s(\xi_x-\tau))e^{-\int^{\xi_x}_{\xi_x-\tau}D(r)\, dr}\right)^2 \right)^{1/2}\\
  &\leq |h-r|\left(\left(\overline DR_0+R_0p(R_0)\right)^2+\left(\overline DR_0+R_0p(R_0)\right)^2 \right)^{1/2}\\
  &=|h-r|\sqrt{2} R_0\left(\overline D+p(R_0)\right)
\end{align*}
for all $h,\,r\geq 0$ and $\xi_s$, $\xi_x$  between  $h$ and $r$.
Therefore 
$$(s,x)\Big|_{[k\omega-\tau,k\omega]}=P^k(\phi)\in S_2,\quad \textnormal{for all }k\in \mathbb{N}.$$

Next, we need to show that there exists  $m\in\N$ such that if $k\geq m$ then $P^k(\phi)\in S_0$.  Let $T\geq  T^{in} $ such that 
$$R_0e^{-\int_0^TD(r)\, dr}\leq \overline s.$$
If $t\geq T$ then, by using again Theorem \ref{persistenciauniforme} and considering (\ref{cotaentau}), it follows that $|x(t)|\geq \delta$.  Estimating $|(s,x)(t)|$ one more time as  in  (\ref{cotas+x}) we have 
\begin{equation*}
  |(s,x)(t)|   \leq R_0e^{-\int_0^tD(r)\, dr}+\overline s\leq 2\, \overline s.
\end{equation*}
As before, it follows that  
$$|(s,x)(h)-(s,x)(r)|\leq |h-r|\sqrt{2} R_0\left(\overline D+p(R_0)\right) $$
for all $h,\,r\geq 0$. Thus, if $k\omega \geq m\omega \geq T$ we obtain  
$$(s,x)\Big|_{[k\omega-\tau,k\omega]}=P^k(\phi)\in S_0.$$
Finally,  by Horn's  fixed point Theorem \cite[Theorem 6]{horn1970some}, there exist a fixed point of $P$ whose second coordinate is positive. 
\end{proof}

\subsection{Proof of Theorem \ref{teoremadeestabilidad}}
\begin{proof}
Consider a solution $(s_1,x_1)$ with no null initial conditions and  $(s_2,x_2)$ a periodic non-trivial solution, whose existence was proved in Theorem \ref{teoremadeexistencia}. 
The proof will be divided into three steps.

\emph{Step 1.} Note that Lemma \ref{funcionz}  yields to $z^*(t-\tau)=(s_2 +x_2+y_2)(t-\tau)$ and then
\begin{align*}
   x_2(t) +\frac{(s_2-z^*)(t-\tau)x_2(t)}{(x_2+y_2)(t-\tau)}=0.
\end{align*}
On another hand, and also using Lemma \ref{lemamuyinteresante}, we get that
\begin{align*}
    x_1(t)+\frac{(s_1-z^*)(t-\tau)x_2(t)}{(x_2+y_2)(t-\tau)}&= x_1(t)+ (s_1-z^*)(t-\tau)e^{-\int_{t-\tau}^t D(r)\,dr}\\
    &=x_1(t)-(x_1+y_1)(t-\tau)e^{-\int_{t-\tau}^t D(r)\,dr}\\
    &\quad + (s_1+x_1+y_1-z^*)(t-\tau)e^{-\int_{t-\tau}^t D(r)\,dr}\\
    &=\left(x_1(\tau)e^{\int_0^\tau D(r)\,dr}-(x_1+y_1)(0)\right)e^{-\int_{0}^tD(r)dr}\\
    &\quad +(s_1+x_1+y_1-z^*)(0)e^{-\int_{0}^t D(r)\,dr}\\
    &=-C_0 e^{-\int_0^t D(r)\,dr},
\end{align*}
 where 
\begin{align*}
    C_0 = (z^*-s_1)(0)-x(\tau)e^{\int_0^\tau D(r)\, dr } .
\end{align*}
 Then we obtain
\begin{equation}\label{diferenciabiomasas}
\begin{aligned}
    (x_1-x_2)(t)&=-(s_1-s_2)(t-\tau)\frac{x_2(t)}{(x_2+y_2)(t-\tau)}\\
    &\quad +\left(\frac{(s_1-z^*)(t-\tau)x_2(t)}{(x_2+y_2)(t-\tau)}+x_1(t)\right.\\
    &\hspace{4cm}\left.-\frac{(s_2-z^*)(t-\tau)x_2(t)}{(x_2+y_2)(t-\tau)}-x_2(t)\right)\\ 
    &=-(s_1-s_2)(t-\tau)\frac{x_2(t)}{(x_2+y_2)(t-\tau)}-C_0 e^{-\int_0^t D(r)\,dr}.
\end{aligned}
\end{equation}
Therefore,
\begin{align*}
  \frac{d}{dt}(s_1-s_2)(t) &=-  D(t)(s_1-s_2)(t)-x_1(t)p(s_1(t))+x_2(t)p(s_2(t)) \\
  &= -  D(t)(s_1-s_2)(t)-x_1(t)(p(s_1(t))-p(s_2(t)))\\
  &\quad -(x_1-x_2)(t)p(s_2(t)) \\
  &= -  D(t)(s_1-s_2)(t)-x_1(t)p'(\xi)(s_1-s_2)(t)\\
  &\quad -(x_1-x_2)(t)p(s_2(t)) \\
  %&= -  D(t)(s_1-s_2)(t)-x_1(t)p'(\xi)(s_1-s_2)(t)\\
  %&\quad -\left(x_1(t)-(x_2+y_2)(t-\tau)e^{-\int_{t-\tau}^tD(r)\,dr} \right)p(s_2(t)) \\
  &= -(s_1-s_2)(t)( D (t)+x_1(t)p'(\xi) )\\
  &\quad +(s_1-s_2)(t-\tau)\frac{x_2(t)p(s_2(t))}{(x_2+y_2)(t-\tau)} +C_0p(s_2(t))e^{-\int_0^tD(r)\,dr} 
\end{align*}
where $\xi=\xi(t)\in (s_1(t),s_2(t))$, but for simplicity we do not write the dependence on $t$.

\emph{Step 2.} Let us define 
$$m=\min_{t\geq \tau}\{x_1(t)p'(\xi(t))\}.$$
Since the system is persistent, $p'$ is positive, and $s_1$, $s_2$ are bounded by above, we deduce that $m>0$. 
Furthermore consider the function
$$J(\varepsilon)=-m+ \max_{h\in[0,\omega]}\left\{\frac{x_2(h)p(s_2(h ))}{(x_2+y_2)(h)}\right\}\left(e^{\tau\varepsilon}-1\right)+3 \varepsilon/2.$$
Note that $J(0)<0$, hence we can fix a positive $\varepsilon$ such that $\varepsilon< \langle D\rangle/2$  and  $J(\varepsilon)\leq 0$.  Also notice that there exists $t_0\geq \omega$ such that   
\begin{equation}\label{Definiciont0}
  C_0\max_{h\in[0,\omega]}\left\{ p(s_2(h ))\right\}e^{- \int_0^{t_0-\omega} D(r)\,dr}\leq \min_{h\in[0,\omega]}\left\{\frac{\varepsilon\,(x_2+y_2)(h)}{2}\right\}. 
\end{equation} 
With all the previous considerations, define
\begin{displaymath}
 w (t)=\max_{h\in[t_0-\tau,t_0]}\left\{\frac{(s_1-s_2)(h)}{(x_2+y_2)(h)},1 \right\}(x_2+y_2)(t)e^{-(t-t_0)\varepsilon},\quad t\geq t_0-\tau,   
\end{displaymath}
and
\begin{displaymath}
  \mathcal{S}=\left\{t\geq t_0:(s_1-s_2)(h)\leq w (h)\text{ for all }h\in[t_0,t]\right\}.   
\end{displaymath}
Observe that $\mathcal{S}$ is a non-empty set  since %, indeed  for $t_0$ we have \textcolor{red}{ $$(s_1-s_2)(t_0)=(s_1-s_2)(t_0)\frac{(x_2+y_2)(t_0)}{(x_2+y_2)(t_0)}\leq w(t_0), $$
%it follows that 
$t_0\in\mathcal{S}$. %}  
We claim that $T^*=\sup\mathcal{S} $ is infinite. On the contrary, suppose that $T^*$ is finite. Consequently, by the mean value Theorem there is $t^*\in[T^*,T^*+\tau)$ such that 
\begin{equation}\label{absurdo}
  \frac{d}{dt}(s_1-s_2)(t^*)\geq \frac{d}{dt}w (t^*),\quad  (s_1-s_2)(t^*)>w (t^*)>0.
\end{equation}
Observe that 
\begin{align*}
    w(t-\tau)&=%\max_{h\in[t_0-\tau,t_0]}\left\{\frac{(s_1-s_2)(h)}{(x_2+y_2)(h)},1 \right\}(x_2+y_2)(t-\tau)e^{-(t-\tau-t_0)\varepsilon}\\
     \max_{h\in[t_0-\tau,t_0]}\left\{\frac{(s_1-s_2)(h)}{(x_2+y_2)(h)},1 \right\}(x_2+y_2)(t-\tau)\frac{(x_2+y_2)(t)}{(x_2+y_2)(t)}e^{-(t-t_0)\varepsilon}e^{\tau\varepsilon}\\
    &=w (t)\frac{(x_2+y_2)(t-\tau)}{(x_2+y_2)(t)}e^{\tau\varepsilon},
\end{align*} 
and that
\begin{align*}
    \frac{d}{dt}w (t)& = \max_{h\in[t_0-\tau,t_0]}\left\{\frac{(s_1-s_2)(h)}{(x_2+y_2)(h)},1 \right\} \\
    &\quad e^{-(t-t_0)\varepsilon} \left(-D(t)(x_2+y_2)(t)-x_2(t)p(s_2(t))-\varepsilon(x_2+y_2)(t) \right)\\
    &=w(t)\left(-D(t)-\frac{x_2(t)}{(x_2+y_2)(t)}p(s_2(t))-\varepsilon\right).
    %&=w(t)\left(-D(t)-e^{-\int_p(s_2(t))-\varepsilon\right).
\end{align*}
Furthermore, using \eqref{Definiciont0}, for $t\geq t_0$,  we have that
\begin{align*}
   C_0p(s_2(t))e^{-\int_0^tD(r)\, dr}&\leq  \min_{h\in[0,\omega]}\left\{\frac{\varepsilon\,(x_2+y_2)(h)}{2}\right\}e^{-\int_{t_0-\omega}^tD(r)\, dr}\\
   &\leq  \frac{\varepsilon}{2}(x_2+y_2)(t) e^{-(t-t_0)\varepsilon}\\
   &\leq \frac{\varepsilon}{2}w(t),
\end{align*}
obtained using the inequality
$$e^{(t-t_0)\varepsilon-\int_{t_0-\omega}^tD(r)\, dr}\leq 1,$$
which holds since $\varepsilon\leq \langle D\rangle /2$.
Now, taking the derivative of $s_1-s_2$ from the previous step and using the above estimations we have   
\begin{align*}
  \frac{d}{dt}(s_1-s_2)(t^*)&\leq -(s_1-s_2)(t^*)( D(t^*) +m)+(s_1-s_2)(t^*-\tau) \frac{x_2(t^*)p(s_2)(t^*) }{(x_2+y_2)(t^*-\tau)}\\
  &\quad +C_0p(s_2(t^*))e^{-\int_0^{t^*}D(r)\, dr}\\
  &<-w (t^*)(D(t^*)+m) +w (t^*-\tau) \frac{x_2(t^*)p(s_2)(t^*) }{(x_2+y_2)(t^*-\tau)}+ \frac{\varepsilon}{2} w(t^*) \\
  &=w(t^*)\left(- D(t^*) -m+ \frac{x_2(t^*)p(s_2)(t^*) }{(x_2+y_2)(t^*)}e^{\tau\varepsilon}+ \frac{\varepsilon}{2} \right) \\
 % &=w(t^*)\left(- D(t^*) + \frac{x_2(t^*)p(s_2)(t^*) }{(x_2+y_2)(t^*)}-m\right.\\
  %&\hspace{1cm}\left.+ \frac{x_2(t^*)p(s_2)(t^*) }{(x_2+y_2)(t^*)}(e^{\tau\varepsilon}-1)+ \frac{\varepsilon}{2} \right).
%\end{align*} 
%From the definition of $\varepsilon$ it follows that 
%\begin{align*}
  %\frac{d}{dt}(s_1-s_2)(t^*)
  &< w(t^*)\left(- D(t^*) + \frac{x_2(t^*)p(s_2)(t^*) }{(x_2+y_2)(t^*)}-\varepsilon \right) \\
  &=\frac{d}{dt}w(t^*)
\end{align*}
where we use the definition of $\varepsilon$. This contradicts (\ref{absurdo}) and, therefore,  $\sup \mathcal{S}=\infty$.

\emph{Step 3.} Following a similar reasoning, this time for $s_2-s_1$, we obtain
\begin{align*}
    \frac{d}{dt}(s_2-s_1)(t)&\leq - (s_2-s_1)(t)( D(t) +x_1(t)p'(\xi) )\\
    &\quad +(s_2-s_1)(t-\tau) \frac{x_2(t)p(s_2(t))}{(x_2+y_2)(t-\tau)}+\left(-C_0\right)p(s_2(t))e^{-\int_0^tD(r)\, dr}
\end{align*}
and we reach the desired result due to
$$|(s_1-s_2)(t)|\leq  \max_{h\in[t_0-\tau,t_0]}\left\{\frac{(s_1-s_2)(h)}{(x_2+y_2)(h)},1 \right\}(x_2+y_2)(t)e^{-(t-t_0)\varepsilon}$$
which tends to zero as $t$ tends to infinity (note that we do not know the sign of $C_0$).

It remains to prove that $x_1(t)$ tends to $x_2(t)$ when $t$ goes to infinity.
By Equation (\ref{diferenciabiomasas}) we obtain
\begin{align*}
    |(x_1-x_2)(t)|&=\left|(s_1-s_2)(t-\tau) e^{-\int_{t-\tau}^t D(r)\, dr}+C_0e^{-\int_0^t D(r)\, dr}\right|\\
    &\leq \left|(s_1-s_2)(t-\tau) \right|+|C_0|e^{-\int_0^t D(r)\, dr}
\end{align*}
and we conclude that $|(x_1-x_2)(t)|$ tends to zero as $t$ tends to infinity  with exponential decay.
% \textcolor{red}{ In turn, from (\ref{eq: 5.11}),  by using Lemma \ref{funcionz} and the previous inequality we deduce that  
% $|(x_1+y_1-x_2-y_2)(t-\tau)|\to 0$
%  as $t$ tends to infinity, with exponential decay. Next, observe that a consequence of Lemma \ref{lemamuyinteresante} is that
% $$\left|(x_1-x_2)(t)-(x_1+y_1-x_2-y_2)(t-\tau)e^{-\int_{t-\tau}^tD(r)\, dr}\right|\to 0$$
% when $t$ tends to infinity. Therefore, we conclude that $|(x_1-x_2)(t)|$ tends to zero as $t$ tends to infinity, with exponential decay.}

Finally, for the uniqueness of the periodic solution, it suffices to note that the distance between two periodic solutions can only tend to zero if it is zero at all times.

\end{proof}

\section*{Acknowledgments}  %    This work was partially supported by  Conicet  under grant PIP 11220200100175CO and by project TOMENADE [MATH-AmSud, 21-MATH-08].
This research is partially supported by  
PROGRAMA REGIONAL MATH-AMSUD MATH2020006. The first author is supported by CONICET  under grant PIP 11220200100175CO. 
The second author is supported  by FONDECYT 11190457.

\addcontentsline{toc}{chapter}{Bibliograf\'ia}
\bibliography{biblio}
\bibliographystyle{plain}
\end{document}